\documentclass[12pt,reqno]{amsart}
\usepackage{mathrsfs}
\usepackage{amsmath}
\usepackage{}
\usepackage{}
\usepackage{amsfonts}
\usepackage{a4wide}
\usepackage{amsmath,amsthm,amssymb,amscd}
\usepackage{latexsym}
\usepackage{hyperref}
\usepackage[numbers,sort&compress]{natbib}
\usepackage{hypernat}
\allowdisplaybreaks
\numberwithin{equation}{section}

\newtheorem{theorem}{Theorem}[section]
\newtheorem{proposition}[theorem]{Proposition}

\newtheorem{lemma}[theorem]{Lemma}

\theoremstyle{definition}

\newtheorem{remark}[theorem]{Remark}

\newcommand{\va}{\varepsilon}

\newcommand{\ds}{\displaystyle}

\newcommand{\sik}{\sum_{i=1}^k}

\newcommand{\be}{\beta}
\newcommand{\al}{\alpha}

\def\r{\mathbb{R}}

\begin{document}

\title[On nonlinear fractional
Schr\"{o}dinger equations]
{Infinitely many positive solutions for  nonlinear  fractional
Schr\"{o}dinger equations}

 \author{Wei Long, \,\,Shuangjie Peng \,\,and \,\,Jing Yang}

 \address{School of Mathematics and Statistics, Central China Normal
University, Wuhan, 430079, P. R. China\newline
 \noindent\quad College of Mathematics
and Information Science, Jiangxi Normal University, Nanchang,
Jiangxi 330022, P. R. China }

\email{ hopelw@126.com}

\address{School of Mathematics and Statistics, Central China
Normal University, Wuhan, 430079, P. R. China }

\email{ sjpeng@mail.ccnu.edu.cn}

\address{School of Mathematics and Statistics, Central China
Normal University, Wuhan, 430079, P. R. China }

\email{ yyangecho@163.com}

\address{School of Mathematics and Statistics, Central China
Normal University, Wuhan, 430079, P. R. China }

\begin{abstract}

We consider the following nonlinear fractional Schr\"{o}dinger  equation
$$
(-\Delta)^su+u=K(|x|)u^p,\ \ u>0 \ \ \hbox{in}\ \  \r^N,
$$
where $K(|x|)$ is a positive radial function, $N\ge 2$, $0<s<1$,
$1<p<\frac{N+2s}{N-2s}$.  Under some asymptotic assumptions on $K(x)$ at infinity, we
show that this problem  has
infinitely many non-radial positive solutions, whose energy can be made arbitrarily large.


\end{abstract}

\maketitle

\section{Introduction}

\ \ \ \ In this paper, we consider the following nonlinear fractional
Schr\"{o}dinger equation
\begin{equation} \label{eq}
(-\Delta)^su+u=K(x)u^p,\ \ u>0\ \ \hbox{in} \ \ \r^N,\ \ u \in\ H^s(\r^N)
\end{equation}
with dimension $N \geq 2$, where $K(x)$ is a positive continuous potential, $0<s<1$,  $1<p<2_*(s)-1,\,\,2_*(s)=\frac{2N}{N-2s}$ and
$$
H^s(\r^N):=\Big\{u\in L^2(\r^N): \frac{|u(x)-u(y)|}{|x-y|^{\frac{N}{2}+s}}\in L^2(\r^N\times\r^N)\Big\}.
$$
The fractional Laplacian of a function $f:\r^N\rightarrow \r$ is
expressed by the formular
\begin{equation} \label{1.3}
(-\Delta)^sf(x)=C_{N,s}P.V.\int_{\r^N}\frac{f(x)-f(y)}{|x-y|^{N+2s}}dy=C_{N,s}\lim_{\va\rightarrow
0}\int_{\r^N\setminus B_\va(x)}\frac{f(x)-f(y)}{|x-y|^{N+2s}}dy,
\end{equation}
where
$C_{N,s}$ is some normalization constant (See Sect. 2).

Problem \eqref{eq} arises from looking for standing waves
$\Psi(t,x)=\exp(iE t)u(x)$ for the following nonlinear
Schr\"{o}dinger equations
\begin{equation} \label{1.4}
i\frac{\partial \Psi}{\partial t}=(-\Delta)^s
\Psi+(1+E)\Psi-K(x)|\Psi|^{p-1}\Psi,\,\,(t,x)\in \r^+\times\r^N,
\end{equation}
where $i$ is the imaginary unit. This equation is of
particular interest in fractional quantum mechanics for the study of
particles on stochastic fields modelled by L\'{e}vy processes. A
path integral over the L\'{e}vy flights paths and a fractional
Schr\"{o}dinger equation of fractional quantum mechanics are
formulated by Laskin \cite{l1} from the idea of Feynman and Hibbs's
path integrals (see also \cite{l2}).

The L\'{e}vy processes occur widely in physics, chemistry and
biology. The stable L\'{e}vy processes that give rise to equations
with the fractional Laplacians have recently attracted much research
interest, and there are a lot of results in the literature on the
existence of such solutions, e.g.,
\cite{ct,cdls,fqt,sv,cs,abfs,w,mmt,ly,cgnt,clo} and the references
therein.

A partner problem of \eqref{eq} is the following scalar
field equation
\begin{equation} \label{1.5}
(-\Delta)^su+V(x)u=|u|^{p-1}u,\ \  \hbox{in} \ \ \r^N,\ \ u \in\ H^s(\r^N).
\end{equation}
In the sequel, we will assume that $V,K$ is bounded, and
$V(x)\geq V_0>0,$ $K(x)\geq K_0>0$.

It is known, but not completely trivial, that $(-\Delta)^s$ reduces to the standard Laplacian $-\Delta$ as $s\to 1$. When $s=1$, the classical nonlinear Schr\"{o}dinger equation has
been extensively studied in the last thirty  years. If
\begin{equation} \label{1.6}
\inf_{x\in \r^N} V(x)<\lim_{|x|\rightarrow \infty} V(x),\ \  (\hbox{or}\
\sup_{x\in \r^N}K(x)>\lim_{|x|\rightarrow \infty} K(x)),
\end{equation}
then, using the concentration compactness principle
\cite{pl1,pl2}, one can show that \eqref{eq} and \eqref{1.5} has a
least energy solution. See for example \cite{dn,pl1,pl2}. But if
\eqref{1.6} does not hold, \eqref{eq} or \eqref{1.5} may not have a
least energy solution. So, in this case, one naturally  needs to find solutions with higher
energy. Recently, Cerami et al. \cite{cds} showed that  problem~\eqref{1.5} with $s=1$
has infinitely many sign-changing solutions if $V(x)$ goes to its
limit at infinity from below at a suitable rate.  In \cite{wy1},
Wei and Yan gave a surprising  result which says that \eqref{eq} or \eqref{1.5} with $s=1$ and $V(x)$ or $K(x)$ being radial has solutions with
large number of bumps near infinity and the energy of this solutions
can be very large. This kind of results was generalized  by Ao and Wei in \cite{AW} very recently  to the  case in which $V(x)$ or $K(x)$ does not have any symmetry assumption. For more results on \eqref{eq} and \eqref{1.5} with $s=1$, we can refer to \cite{cds,cp,ds,dn} and the references therein.

When $0<s<1,$ Chen and Zheng \cite{cz} studied  the following singularly perturbed problem
\begin{equation} \label{1.7}
\varepsilon^{2s}(-\Delta)^su+V(x) u=|u|^{p-1}u,\ \ \hbox{in} \ \ \r^N.
\end{equation}
They showed that when $N=1,2,3,$ $\varepsilon$ is sufficiently small,
$\max\{\frac12,\frac n4\}<s<1$ and $V$ satisfies some smoothness and boundedness assumptions, equation
\eqref{1.7} has a nontrivial solution $u_{\varepsilon}$
concentrated  to some single point as $\varepsilon\rightarrow0.$ Very recently, in
\cite{dpw}, D\'{a}vila, del Pino and Wei generalized various existence results known for \eqref{1.7} with $s=1$ to the case of fractional Laplacian. For results which are not for singularly perturbed type of \eqref{eq} and \eqref{1.5} with $0<s<1$, the readers can refer to \cite{ct,fqt,se1,ta} and the references therein.

As far as we know, it seems that there is no result on the existence of multiple solutions of equation \eqref{eq}
which is not a singularly perturbed problem. The aim of
this paper is to obtain infinitely many non-radial positive
solutions for \eqref{eq} whose functional energy are very large, under some assumptions for
 $K(x)=K(|x|)>0$ near the infinity.

Let
\begin{equation}\label{1.7"}
\frac{N+2s}{N+2s+1}<m< N+2s.
\end{equation}

 We assume that $0<K(|x|)\in C(\r^N)$
satisfies the following condition at infinity
$$
(K):\,\,\, K(r)=K_0-\frac{a}{r^m}+O(\frac{1}{r^{m+\theta}}),\,\,\hbox{as}\,\,r\rightarrow +\infty,
$$
 for some $a>0,\,\,\theta>0$. Without loss of
generality, we may assume that $K_0=1$.

 Our main
result in this paper can be stated as follows
\begin{theorem}\label{th1}
 Suppose that $N\ge 2,\,\,0<s<1$,  $1<p<2_*(s)-1$. If $K(r)$ satisfies ($K$), then problem \eqref{eq} has
infinitely many non-radial positive solutions.
\end{theorem}

\begin{remark}The radial symmetry can be replaced by the following
weaker symmetry assumption: after suitably rotating the coordinated
system,

(i) $K(x)=K(x',x'')=K(|x'|,|x_3|,\cdots,|x_N|),$ where
$x=(x',x'')\in \r^2\times \r^{N-2}$,

(ii) $K(x)=K_0-\frac{a}{|x|^m}+O(\frac{1}{|x|^{m+\theta}})$ as
$|x|\rightarrow +\infty$, where $a>0,\
\theta>0$ and $K_0>0$ are some constants.
\end{remark}

\begin{remark}Using the same argument, we can prove that if
$$
V(x)=V(x',x'')=V(|x'|,|x_3|,\cdots,|x_N|)=V_0+\frac{a}{|x|^m}+O(\frac{1}{|x|^{m+\theta}}), \,\,\hbox{as}\,\,|x|\to\infty,
$$
for some constants  $V_0>0,$ $a>0,$ and $\theta>0$. Then problem~\eqref{1.5}
has infinitely many positive non-radial solutions.
\end{remark}

Before we close this introduction, let us outline the main idea in
the proof of Theorem \ref{th1}.

 We will use the unique ground state $U$  of
\begin{equation}\label{eq1}
 (-\Delta)^s u+ u=u^p,\ u>0,\,\,x \in \r^N,\,\,u(0)=\max\limits_{\r^N} u(x)
\end{equation}
to build up the approximate solutions for \eqref{eq}. It is well
known that when $s=1$, the ground state solution of \eqref{eq1}
decays exponentially at infinity. But from \cite{fl,fls}, we see that when $s\in(0,1)$, the
unique ground solution of \eqref{eq1} decays like
$\frac{1}{|x|^{N+2s}}$ when $|x|\rightarrow \infty.$

Let any integer $k>0$, define
$$
x^i=\big(r\cos\frac{2(i-1)\pi}{k},r\sin\frac{2(i-1)\pi}{k},0\big),\
\ i=1,\cdots,k,
$$
where $0$ is the zero vector in $\r^{N-2}$,
$r\in[r_0k^{\frac{N+2s}{N+2s-m}},r_1k^{\frac{N+2s}{N+2s-m}}]$ for some
$r_1>r_0>0.$

 Set $x=(x',x''),\ x'\in \r^2,\ x'' \in \r^{N-2}.$ Define
\begin{equation*} \begin{split}
\mathscr{H}=\Bigl\{u: &u\in H^s(\r^N), u\ \mbox{is even in}\ x_i, i=2,\cdots,N,\\
&u(r\cos\theta, r\sin\theta,x'')=
u(r\cos(\theta+\frac{2j\pi}{k}),r\sin(\theta+\frac{2j\pi}{k}),x''),\,j=1,\cdots,k-1\Bigr\}.\\
\end{split} \end{equation*}

Write $$U_r(x)=\sum_{i=1}^{k}U_{x^i}(x),$$ where
$U_{x^i}(x)=U(x-x^i).$

We will prove Theorem \ref{th1} by proving the following result
\begin{theorem}\label{th3}
Under the assumption of Theorem~\eqref{th1}, there is an
integer $k_0>0$, such that for any integer $k\geq k_0$, \eqref{eq}
has a solution $u_k$ of the form
$$
u_k=U_r(x)+\omega_r
$$
where $\omega_r\in \mathscr{H}$,
$r\in[r_0k^{\frac{N+2s}{N+2s-m}},r_1k^{\frac{N+2s}{N+2s-m}}]$ and as
$k\rightarrow +\infty$,
$$
\int_{\r^{2N}}\frac{|\omega_r(x)-\omega_r(y)|^2}{|x-y|^{N+2s}}+\int_{\r^N}\omega_r^2\rightarrow
0.
$$

\end{theorem}

The idea of our proof is inspired by that of \cite{wy1} where
infinitely many positive non-radial solutions  to a nonlinear
Schr\"{o}dinger equations \eqref{1.4} with $s=1$ are obtained when the potential approaches
to a positive constant algebraically at infinity. We will use the well-known Lyapunov-Schmidt
reduction scheme to transfer our problem to a maximization problem  of a one-dimensional function in a suitable range. Compared with the operator $-\Delta$, which is local, the operator $(-\Delta)^s$ with $0<s<1$ on $\r^N$ is nonlocal. So it is expected that the standard techniques for $-\Delta$ do not work directly. In particular, when we try to find  spike solutions for \eqref{eq} with $0<s<1$, $(-\Delta)^s$ may kill bumps by averaging on the whole $\r^N$. For example, the ground state for \eqref{eq} with $0<s<1$   decays
algebraically at infinity, which is a contrast to the fact that the ground state for $-\Delta$ decays exponentially at infinity. This kind of property requires us to  establish some new basic estimates and give a precise estimate on  the energy of the approximate solutions.

This paper is organized as follows. In Sect.2, we will give some
preliminary properties related to the fractional Laplacian operator. In Sect.3, we will establish some preliminary estimate. We will carry out a
reduction procedure and study the reduced one dimensional problem to
prove Theorems ~\ref{th3} in Sect.4. In Appendix,
some basic estimates  and an energy expansion for the functional
corresponding to problem \eqref{eq} will be established.

 \section{Basic Theory on Fractional Laplacian Operator }
 \bigskip
In this section, we recall some properties of the fractional order
Sobolev space and the ground state solution $U$ of the limit
equation \eqref{eq1}.

Let $0<s<1$. Various definitions of the fractional Laplacian
$(-\Delta)^sf(x)$ of a function $f$ defined in $\r^N$ are available,
depending on its regularity and growth properties.

It can be defined as a pseudo-differential operator
$$
\widehat{(-\Delta)^s}f(\xi)=|\xi|^{2s}\widehat{f}(\xi),
$$
where\,\,\, $\widehat{}$\,\,\ is Fourier transform. When $f$ have some
sufficiently regular, the fractional Laplacian of a function
$f:\r^N\rightarrow \r$ is expressed by the formular
\begin{equation*}
(-\Delta)^sf(x)=C_{N,s}P.V.\int_{\r^N}\frac{f(x)-f(y)}{|x-y|^{N+2s}}dy=C_{N,s}\lim_{\va\rightarrow
0}\int_{\r^N\setminus B_\va(x)}\frac{f(x)-f(y)}{|x-y|^{N+2s}}dy,
\end{equation*}
where $C_{N,s}=\pi^{-(2s+N/2)}\frac{\Gamma(N/2+s)}{\Gamma(-s)}$. This integral makes sense
directly when $s<\frac12$ and $f\in C^{0,\al}(\r^N)$ with $\al>2s$,
or if $f\in C^{1,\al}(\r^N)$, $1+2\al>2s.$ It is well known that $(-\Delta)^s$ on $\r^N$ with $0<s<1$ is a  nonlocal operator. In the remarkable
  work of Caffarelli and Silvestre \cite{cs}, this nonlocal operator was expressed as a generalized Dirichlet-Neumann map for a certain elliptic boundary
    value problem with nonlocal differential operator defined on the upper half-space $\r_+^{N+1}:=\{(x,y):\,x\in\r^N,\,y>0\}$. That is, for a function $f:\r^N\rightarrow \r$, we consider the
extension $u(x,y):\r^N\times[0,+\infty)\rightarrow \r$ that satisfies the
equation
\begin{equation} \label{2.2}
u(x,0)=f(x)
\end{equation}
\begin{equation} \label{2.3}
\Delta_x u+\frac{1-2s}{y}u_y+u_{yy}=0.
\end{equation}
The equation \eqref{2.3} can also be written as
\begin{equation} \label{2.4}
div(y^{1-2s}\nabla u)=0,
\end{equation}
which is clearly the Euler-Lagarnge equation for the functional
$$
J(u)=\int_{y>0}|\nabla u|^2y^{1-2s}dxdy.
$$
Then, it follows from \cite{cs} that
$$
C(-\Delta)^sf=\lim_{y\rightarrow
0^+}-y^{1-2s}u_y=\frac{1}{2s}\lim_{y\rightarrow
0^+}\frac{u(x,y)-u(x,0)}{y^{2s}}.
$$

When $s\in (0,1)$, the space $H^s(\r^N)=W^{s,2}(\r^N)$ is defined by
\begin{eqnarray*}
H^s(\r^N)&=&\Big\{u\in
L^2(\r^N):\frac{|u(x)-u(y)|}{|x-y|^{\frac{N}{2}+s}}\in
L^2(\r^N\times\r^N)\Big\}\\
&=&\Big\{u\in
L^2(\r^N):\int_{\r^N}(1+|\xi|^{2s})|\widehat{u}(\xi)|^2<+\infty\Big\}
\end{eqnarray*}
and the norm is
$$
\|u\|_s:=\|u\|_{H^s(\r^N)}=\Big(\int_{\r^N}\int_{\r^N}\frac{|u(x)-u(y)|^2}{|x-y|^{N+2s}}dxdy+\int_{\r^N}|u|^2dx\Big)^{\frac12}.
$$
Here the term
$$
[u]_s:=\Big(\int_{\r^N}\int_{\r^N}\frac{|u(x)-u(y)|^2}{|x-y|^{N+2s}}dxdy\Big)^{\frac12}
$$
is the so-called Gagliardo (semi) norm of $u$. The following
identity yields the relation between the fractional operator
$(-\Delta)^s$ and the fractional Laplacian Sobolev space $H^s(\r^N)$,
$$
[u]_{H^s(\r^N)}=C\Big(\int_{\r^N}|\xi|^{2s}|\widehat{u}(\xi)|^2d\xi\Big)^{\frac12}=C\|(-\Delta)^{\frac{s}{2}}u\|_{L^2(\r^N)}
$$
for a suitable positive constant $C$ depending only $s$ and $N$.

Clearly, $\|\cdot\|_{H^s(\r^N)}$ is a Hilbertian norm induced by the
inner product
\begin{eqnarray*}
\langle
u,v\rangle_{H^s(\r^N)}&=&\int_{\r^N}\int_{\r^N}\frac{(u(x)-u(y))(v(x)-v(y))}{|x-y|^{N+2s}}dxdy+\int_{\r^N}u(x)v(x)dx\\
&=&\langle u,v\rangle_s+\langle u,v\rangle_{L^2(\r^N)}.
\end{eqnarray*}

On the Sobolev inequality and the compactness of embedding, one has
\begin{theorem}\label{th2.1}\cite{npv} The following imbeddings are continuous:

(1) $H^s(\r^N)\hookrightarrow L^q(\r^N)$, $2\leq q\leq
\frac{2N}{N-2s},$ if $N>2s,$

(2)$H^s(\r^N)\hookrightarrow L^q(\r^N)$, $2\leq q\leq +\infty,$ if $
N=2s,$

(3)$H^s(\r^N)\hookrightarrow C_b^j(\r^N)$, if $N<2(s-j)$.

 Moreover,  for any $R>0$ and $p\in [1,2_*(s))$ the embedding
 $H^s(B_R)\hookrightarrow\hookrightarrow L^p(B_R)$ is compact,
 where
 $$
 C_b^j(\r^N)=\Big\{u\in C^j(\r^N):\,D^Ku \,\text{is bounded  on}\,\ \r^N\,\text{for}\,|K|\leq j\Big\}.
 $$
\end{theorem}

Now,  we recall some known results for the limit
equation \eqref{eq1}.
  If $s=1$, the uniqueness and non-degeneracy of the ground state $U$ for \eqref{eq1} is due to \cite{k}. In the celebrated paper
\cite{fl}, Frank and Lenzemann proved the uniqueness of ground state
solution $U(x)=U(|x|)\geq 0$ for $N=1, 0<s<1,$ $1<p<2_*(s)-1$. Very
recently, Frank, Lenzemann and Silvestre \cite{fls} obtained the
non-degeneracy of ground state solutions for \eqref{eq1} in arbitrary
dimension $N\geq1$ and any admissible exponent $1<p<2_*(s)-1$.

  For convenience, we summarize the properties of the ground state $U$  of \eqref{eq1} which can be found in \cite{fl,fls}.
\begin{theorem}\label{th2.2}Let $N\geq1,\ s\in (0,1)$ and
$1<p<2_*(s)-1$. Then the following hold.

(i) (Uniqueness)  The ground state solution $U\in H^s(\r^N)$ for
equation \eqref{eq1} is unique.

(ii)(Symmetry, regularity and decay)  $U(x)$ is radial, positive and
strictly decreasing in $|x|$. Moreover, the function $U$ belongs
to $H^{2s+1(\r^N)}\cap C^{\infty}(\r^N)$ and satisfies
$$
\frac{C_1}{1+|x|^{N+2s}}\leq U(x) \leq \frac{C_2}{1+|x|^{N+2s}}\ \ \
for\ \ x\in\r^N,
$$
with some constants $C_2\geq C_1>0.$

(iii) (Non-degeneracy) The linearized operator
$L_0=(-\Delta)^s+1-p|U|^{p-1}$ is non-degenerate, i.e., its kernel
is given by
$$
kerL_0=span\{\partial_{x_1}U,\partial_{x_2}U,\cdots
\partial_{x_N}U\}.
$$
\end{theorem}
 By Lemma C.2 of \cite{fls}, it holds
that, for $j=1,\cdots,N,$ $\partial_{x_j}U$ has the
following decay estimate,
$$
|\partial_{x_j}U|\leq \frac{C}{1+|x|^{N+2s}}.
$$
From Theorem \ref{th2.2}, the ground bound state solution like
$\frac{1}{|x|^{N+2s}}$ when $|x|\rightarrow +\infty.$ Fortunately,
this polynomial decay is enough for us in the
estimates of our proof.
 \section{Some Preliminaries}
Let
$$
Z_i=\frac{\partial U_{x^i}}{\partial r},\ \ \ i=1,\cdots,k,
$$
where $x^i=(r\cos\frac{2(i-1)\pi}{k},r\sin\frac{2(i-1)\pi}{k},0)$ and
$$
r\in[r_0k^{\frac{N+2s}{N+2s-m}},r_1k^{\frac{N+2s}{N+2s-m}}]
$$
for some
$r_1>r_0>0.$

Define
$$
E=\Big\{v\in \mathscr{H}: \sik\int_{\r^N}U_{x^i}^{p-1}Z_iv=0\Big\}.
$$
The norm of $H^s(\r^N)$ is defined by:
$$
\|v\|_s=\sqrt{\langle v,v\rangle},\ \ \ v\in H^s(\r^N),
$$
where
\begin{eqnarray*}
\langle v_1,v_2\rangle&=&\langle
v_1,v_2\rangle_s+\int_{\r^N}v_1v_2\\
&=&\int_{\r^N}\int_{\r^N}\frac{(u(x)-u(y))(v(x)-v(y))}
{|x-y|^{N+2s}}dxdy+\int_{\r^N}u(x)v(x)dx.
\end{eqnarray*}

 The variational functional corresponding to \eqref{eq} is
\begin{eqnarray*}
I(u)= \frac{1}{2}\langle
u,u\rangle_s+\frac{1}{2}\int_{\r^N}u^2-\frac{1}{p+1}\int_{\r^N}K(x)|u|^{p+1}.
\end{eqnarray*}
Let
$$
J(\varphi)=I(U_r+\varphi)=I\Big(\sik U_{x^i}+\varphi\Big),\ \
\varphi\in E.
$$
Expand $J(\varphi)$ as follows:
\begin{equation}\label{expand}
J(\varphi)=J(0)+l(\varphi)+\frac{1}{2}\langle
L(\varphi),\varphi\rangle+R(\varphi),\ \ \ \varphi\in E,
\end{equation}
where
\begin{equation*}
l(\varphi)=\int_{\r^N}\sik U_{x^i}^p\varphi-\int_{\r^N}K(x)\Big(\sik
U_{x^i}\Big)^{p}\varphi,
\end{equation*}

\begin{equation*}
\langle
L(\varphi),\varphi\rangle=\langle\varphi,\varphi\rangle_s+\int_{\r^N}(\varphi^2-pK(x)U_r^{p-1}\varphi^2)
\end{equation*}
and
\begin{equation*}
R(\varphi)=-\frac{1}{p+1}\int_{\r^N}K(x)\Big((U_r+\varphi)^{p+1}-U_r^{p+1}-(p+1)U_r^p\varphi-\frac{1}{2}(p+1)pU_r^{p-1}\varphi^2\Big).
\end{equation*}
In order to find a critical point for $J(\varphi)$, we need to
discuss each term in the expansion \eqref{expand}.

\begin{lemma}\label{lm3.2}There is a constant $C>0$ independent of $k$ such
that
$$
\|R'(\varphi)\|\leq C \|\varphi\|_s^{\min\{p,2\}},
$$
$$
\|R''(\varphi)\|\leq C \|\varphi\|_s^{\min\{p-1,1\}}.
$$
\end{lemma}
\begin{proof}
By direct calculation, we know that
$$
\langle
R'(\varphi),\psi\rangle=-\int_{\r^N}K(x)\Big((U_r+\varphi)^p-U_r^p-pU_r^{p-1}\varphi\Big)\psi,
$$
$$
\langle
R''(\varphi)(\psi,\xi)\rangle=-p\int_{\r^N}K(x)\Big((U_r+\varphi)^{p-1}-U_r^{p-1}\Big)\psi\xi.
$$

Firstly, we deal with the case $p>2.$ Since
$$
|\langle R'(\varphi),\psi\rangle|\leq
C\int_{\r^N}U_r^{p-2}|\varphi|^2|\psi|\leq
C\Big(\int_{\r^n}(U_r^{p-2}|\varphi|^2)^{\frac{p+1}{p}}\Big)^{\frac{p}{p+1}}\|\psi\|_s,
$$
we find
$$
\|R'(\varphi)\|\leq
C\Big(\int_{\r^n}(U_r^{p-2}|\varphi|^2)^{\frac{p+1}{p}}\Big)^{\frac{p}{p+1}}.
$$
On the other hand, it follows from Lemma \ref{lmA.2} that $U_r$ is
bounded. Since $2<\frac{2(p+1)}{p}<p+1,$ we obtain
$$
\|R'(\varphi)\|\leq
C\Big(\int_{\r^n}|\varphi|^{\frac{2(p+1)}{p}}\Big)^{\frac{p}{p+1}}\leq
C\|\varphi\|_s^2.
$$
For the estimate of $\|R''(\varphi)\|$, we have
\begin{eqnarray*}
|R''(\varphi)(\psi,\xi)|&\leq&
C\int_{\r^N}U_r^{p-2}|\varphi||\psi||\xi|\\
&\leq&C\int_{\r^N}|\varphi||\psi||\xi|\leq
C\Big(\int_{\r^N}|\varphi|^3\Big)^{\frac{1}{3}}\Big(\int_{\r^N}|\psi|^3\Big)^{\frac{1}{3}}\Big(\int_{\r^N}|\xi|^3\Big)^{\frac{1}{3}}\\
&\leq&C\|\varphi\|_s\|\psi\|_s\|\xi\|_s.
\end{eqnarray*}
So,
$$
\|R''(\varphi)\|\leq C\|\varphi\|_s.
$$
Using  the same argument, for the case $1<p\leq 2$,  we also can obtain that
$$
\langle R'(\varphi),\psi\rangle\leq C\int_{\r^N}\varphi^p\psi\leq
C\|\varphi\|_s^p\|\psi\|_s,
$$
$$
\langle R'(\varphi)(\psi,\xi)\rangle\leq
C\int_{\r^N}\varphi^{p-1}\psi \xi \leq
C\|\varphi\|_s^{p-1}\|\psi\|_s\|\xi\|_s.
$$
\end{proof}
\section{The Finite-Dimensional reduction and proof of the main results}
In this section, we intend to prove the main theorem by the
Lyapunov-Schmidt reduction.

Associated to the quadratic form  $L(\varphi)$, we define $L$ to be
a bounded linear map from $E$ to $E$, such that
$$
\langle Lv_1,v_2\rangle=\langle
v_1,v_2\rangle_s+\int_{\r^N}\Big(v_1v_2-pK(x)U_r^{p-1}v_1v_2\Big),\
\ \ \ v_1,v_2\in E.
$$
In this paper, we always assume
\begin{equation}
r\in
S_k:=\Big[\Big(\Big(\frac{B_0(N+2s)}{B_1m}-\alpha\Big)^{\frac{1}{N+2s-m}}\Big)k^{\frac{N+2s}{N+2s-m}},
\Big(\Big(\frac{B_0(N+2s)}{B_1m}+\alpha\Big)^{\frac{1}{N+2s-m}}\Big)
k^{\frac{N+2s}{N+2s-m}}\Big],
\end{equation}
where $\alpha>0$ is a small constant, $B_0$ and $B_1$ are defined in
Proposition \ref{propA.1}.

 Next, we show the invertibility of $L$ in $E$.

\begin{proposition}\label{prop4.1} There exists an integer  $k_0>0$, such that for $k\ge k_0$, there is a constant $\rho>0$
independent of $k,$ satisfying that  for any $r\in S_k,$
$$
\|Lu\|\geq \rho\|u\|_s,\ \ \ u\in E.
$$
\end{proposition}
\begin{proof}We argue by contradiction. Suppose that there are $n\rightarrow +\infty,\ r_k\in
S_k$, and $u_n\in E,$ such that
$$
\|Lu_n\|=o(1)\|u_n\|_s,\ \ \ \ \|u_n\|_s^2=k.
$$
Recall
$$
\Omega_i=\Big\{x=(x',x'')\in \r^2\times
\r^{N-2}:\langle\frac{x'}{|x'|},\frac{x^i}{|x^i|}\rangle\geq
\cos\frac{\pi}{k} \Big\},\ \ i=1,2,\cdots,k.
$$
By symmetry, we have
\begin{equation}\label{4.2}\begin{split}
&\int_{\Omega_1}\int_{\r^N}\frac{(u_n(x)-u_n(y))(\varphi(x)-\varphi(y))}{|x-y|^{N+2s}}+\int_{\Omega_1}\Big(u_n\varphi-pK(x)U_{r_k}^{p-1}u_n\varphi\Big)\\
&=\frac{1}{k}\langle Lu_n,\varphi\rangle=o(\frac{1}{\sqrt{k}})\|\varphi\|_s, \ \ \varphi\in E.\\
\end{split}
\end{equation}
In particular,
$$
\int_{\Omega_1}\int_{\r^N}\frac{|u_n(x)-u_n(y)|^2}{|x-y|^{N+2s}}+\int_{\Omega_1}\Big(|u_n|^2-pK(x)U_{r_k}^{p-1}|u_n|^2\Big)=o(1)
$$
and
$$
\int_{\Omega_1}\int_{\r^N}\frac{|u_n(x)-u_n(y)|^2}{|x-y|^{N+2s}}+\int_{\Omega_1}|u_n|^2=1.
$$
 Set $\widetilde{u}_n(x)=u_n(x-x^1).$ Then for any
$R>0,$ since $dist(x^1,\partial \Omega_1)=r\sin \frac{\pi}{k}\geq
Ck^{\frac{m}{N+2s-m}} \rightarrow +\infty,$ $B_R(x^1)\subset \Omega_1$, i.e.
$$
\int_{B_R(x^1)}\int_{\r^N}\frac{|u_n(x)-u_n(y)|^2}{|x-y|^{N+2s}}+\int_{B_R(x^1)}|u_n|^2\leq
1.
$$
One can obtain
$$
\int_{B_R(0)}\int_{\r^N}\frac{|\widetilde{u}_n(x)-\widetilde{u}_n(y)|^2}{|x-y|^{N+2s}}+\int_{B_R(0)}|\widetilde{u}_n|^2\leq
1.
$$
So we suppose that there is a $u \in H^s(\r^N),$ such that as
$n\rightarrow +\infty,$
$$
\widetilde{u}_n \rightharpoonup u,\ \ \ \hbox{in} \ \ \ H^s_{loc}(\r^N)
$$
and
$$
\widetilde{u}_n \rightarrow u,\ \ \ \hbox{in} \ \ L^2_{loc}(\r^N).
$$
Since $\widetilde{u}_n$ is even in $x_j,j=2,\cdots,N,$ it is easy to
see that $u$ is even in $x_j,j=2,\cdots,N.$ On the other hand, from
$$
\int_{\r^N}U_{x^1}^{p-1}Z_1u_n=0,
$$
 we get
$$
\int_{\r^N}U^{p-1}\frac{\partial U}{\partial x_1}\widetilde{u}_n=0.
$$
So, $u$ satisfies
\begin{equation}\label{4.3}
\int_{\r^N}U^{p-1}\frac{\partial U}{\partial x_1}u=0.
\end{equation}
Now, we claim that $u$ satisfies
\begin{equation}\label{4.4}
(-\Delta)^s u +u-pU^{p-1}u=0\ \ \ \hbox{in} \ \ \r^N.
\end{equation}
Indeed, we set
$$
\widetilde{E}=\Big\{\varphi: \varphi \in H^s(\r^N),\
\int_{\r^N}U^{p-1}\frac{\partial U}{\partial x_1}\varphi=0\Big\}.
$$
For any $R>0,$ let $\varphi \in C_0^{\infty}(B_R(0))\cap
\widetilde{E}$ be any function, satisfying that $\varphi$ is even in
$x_j,j=2,\cdots,N.$  Then $\varphi_1(x)=\varphi(x-x^1)\in
C_0^{\infty}(B_R(x^1)).$ We may identify $\varphi_1(x)$ as elements
in $E$ by redefining the values outside $\Omega_1$ with the symmetry.
By using \eqref{4.2} and Lemma \ref{A.2}, we find
\begin{equation}\label{4.5}
\int_{\r^N}\int_{\r^N}\frac{(u(x)-u(y))(\varphi(x)-\varphi(y))}{|x-y|^{N+2s}}+\int_{\r^N}\Big(
u\varphi-pU^{p-1}u\varphi\Big)=0.
\end{equation}
On the other hand, since $u$ is even in $x_j,j=2,\cdots,N,$
\eqref{4.5} holds for any $\varphi \in C_0^{\infty}(B_R(0))\cap
\widetilde{E}.$ By the density of $C_0^\infty(\r^N)$ in $H^s(\r^N)$,
it is easy to show that
\begin{equation}\label{4.6}
\int_{\r^N}\int_{\r^N}\frac{(u(x)-u(y))(\varphi(x)-\varphi(y))}{|x-y|^{N+2s}}+\int_{\r^N}\Big(
u\varphi-pU^{p-1}u\varphi\Big)=0, \ \ \ \ \forall\ \varphi\in
\widetilde{E}.
\end{equation}
We know $\varphi=\frac{\partial U}{\partial x_1}$ is a solution of
\eqref{4.6}, thus \eqref{4.6} is true for any $\varphi \in
H^s(\r^N).$ One see that $u=0$ because $u$ is even in
$y_i,i=2,\cdots,N$ and \eqref{4.4}. As a result,
$$
\int_{B_R(x^1)}u_n^2=o_n(1),\ \ \ \forall\ R>0,
$$
where $o_n(1)\to 0$ as $n\to +\infty$.

Now, using the Lemma \ref{lmA.2}, we obtained that for
any $1<\eta\leq N+2s,$ there is a constant $C>0,$ such that
$$
U_{r_k}(x)\leq \frac{C}{(1+|x-x^1|)^{N+2s-\eta}},\,\,x\in\Omega_1.
$$
 Thus,
\begin{eqnarray*}
o_n(1)&=&\int_{\Omega_1\times
\r^N}\frac{|u_n(x)-u_n(y)|^2}{|x-y|^{N+2s}}+\int_{\Omega_1}\Big(u_n^2-pK(x)U^{p-1}_{r_k}u^2_n\Big)\\
&=&\int_{\Omega_1}\int_{\r^N}\frac{|u_n(x)-u_n(y)|^2}{|x-y|^{N+2s}}+\int_{\Omega_1}u_n^2\\
&&+C\Big(\int_{B_{\frac12R}(x^1)}+\int_{\Omega_1\setminus
B_{\frac12R}(x^1)}\frac{1}{(1+|x-x^1|)^{N+2s-\eta}}u_n^2\Big)\\
 &\geq&
\frac{1}{2}\Bigl(\int_{\Omega_1}\int_{\r^N}\frac{|u_n(x)-u_n(y)|^2}{|x-y|^{N+2s}}+\int_{\Omega_1}u_n^2\Bigr)+o_n(1)+o_R(1),
\end{eqnarray*}
which is impossible for large $R$.

As a result, we get a contradiction.
\end{proof}
\begin{proposition}\label{prop4.2} There is an integer $k_0>0,$ such
that for each $k\geq k_0,$ there is a $C^1$ map from $S_k$ to
$\mathscr{H}$: $\omega=\omega(r),\ r=|x^1|,$ satisfying $\omega\in
E,$ and
$$
 J'(\omega)\Big|_E=0.
$$
Moreover, there exists a constant $C>0$ independent of $k$ such that
\begin{equation}\label{4.7}
\|\omega\|_s\leq Ck^{\frac12}\Bigl(\Bigl(\frac{k}{r}\Bigr)^{\frac{N+2s}{2}+\tau}+\frac{1}{r^{\frac{m}{2}+\tau}}\Bigr),
\end{equation}
where $\tau>0$ is a small constant.
\end{proposition}
\begin{proof}We will use the contraction theorem to prove it.
By the following Lemma \ref{lm4.1}, $l(\omega)$ is a bounded linear
functional in $E$. We know by Reisz representation theorem that there is
an $l_k\in E,$ such that
$$
l(\omega)=\langle l_k,\omega\rangle.
$$
So, finding a critical point for $J(\omega)$ is equivalent to
solving
\begin{equation}\label{4.8}
l_k+L\omega+R'(\omega)=0.
\end{equation}
By Proposition \ref{prop4.1}, $L$ is invertible. Thus, \eqref{4.8}
is equivalent to
$$
\omega=A(\omega):=-L^{-1}(l_k+R'(\omega)).
$$
Set
$$
\bar S_k:=\Bigl\{\omega\in E:\|\omega\|_s\leq k^{\frac12}\Bigl(\frac{k}{r}\Bigr)^{\frac{N+2s+\tau}{2}}+\frac{k^\frac12}{r^{\frac{m+\tau}{2}}}\Bigr\}.
$$

We shall verify that $A$ is a contraction mapping from $\bar S_k$ to
itself. In fact,  on one hand, for any $\omega\in \bar S_k$, by Lemmas
\ref{lm4.1} and \ref{lm3.2}, we obtain
\begin{eqnarray*}
\|A(\omega)\|&\leq&C(\|l\|+\|R'(\omega)\|)\\
&\leq& C\|l\|+C\|\omega\|_s^{\min\{p,2\}}\\
&\leq& C\Bigl(k^{\frac12}\Bigl(\frac{k}{r}\Bigr)^{\frac{N+2s}{2}+\tau}+\frac{k^\frac12}{r^{\frac{m}{2}+\tau}}\Bigr)
+C\Bigl(k^{\frac12}\Bigl(\frac{k}{r}\Bigr)^{\frac{N+2s+\tau}{2}}+\frac{k^\frac12}{r^{\frac{m+\tau}{2}}}\Bigr)^{\min\{p,2\}} \\
&\leq&k^{\frac12}\Bigl(\frac{k}{r}\Bigr)^{\frac{N+2s+\tau}{2}}+\frac{k^\frac12}{r^{\frac{m+\tau}{2}}}.
\end{eqnarray*}

On the other hand, for any $\omega_1,\omega_2 \in \bar S_k$,
\begin{eqnarray*}
\|A(\omega_1)-A(\omega_2)\|&=&\|L^{-1}R'(\omega_1)-L^{-1}R'(\omega_2)\|\\
&\leq& C\|R'(\omega_1)-R'(\omega_2)\|\\
&\leq& C\|R''(\theta\omega_1+(1-\theta)\omega_2)\|\|\omega_1-\omega_2\|_s\\
&\leq& \left\{\begin{array}{ll}
        C(\|\omega_1\|_s^{p-1}+\|\omega_2\|_s^{p-1})\|\omega_1-\omega_2\|_s,\ \ \ & \hbox{$\text{if}\,\,1<p<2$}\vspace{2mm}\\
         C(\|\omega_1\|_s+\|\omega_2\|_s)\|\omega_1-\omega_2\|_s,\ \ \ & \hbox{$\text{if}\,\,p\geq2$}\\
    \end{array}\right.\\
&\leq& \frac{1}{2}\|\omega_1-\omega_2\|_s.
\end{eqnarray*}

Then the result follows from the contraction mapping theorem. The estimate \eqref{4.7} follows  Lemma~\ref{lm4.1}.

The claim that $\omega(r)$ is continuously differentiable in $r$ can be verified by using the same argument employed to proof Lemma~4.4 in \cite{dpw}.
\end{proof}

\begin{lemma}\label{lm4.1} There is a constant $C>0$ and a small constant $\tau>0$, which are independent of $k$,  such
that
$$
\|l_k\|\leq C\Bigl(k^{\frac12}\Bigl(\frac{k}{r}\Bigr)^{\frac{N+2s}{2}+\tau}+\frac{k^\frac12}{r^{\frac{m}{2}+\tau}}\Bigr),
$$
provided $k\ge k_0$ for some integer $k_0>0$.
\end{lemma}

\begin{proof}
\begin{equation}\label{4.9}\begin{split}
|l(\varphi)|&=\Big|\int_{\r^N}\sik U_{x^i}^p\varphi-K(x)\Big(\sik U_{x^i}\Big)^p\varphi\Big|\\
&\leq C \int_{\r^N}\Big|\Big(\sik U_{x^i}\Big)^p-\sik
U_{x^i}^p\Big||\varphi|+\int_{\r^N}\Big|(K(x)-1)(\sik U_{x^i})^p \varphi\Big|.\\
\end{split}
\end{equation}

From $m>\frac{N+2s}{N+2s+1}$, we deduce that
$$
\frac{N+2s-m}{m}<N+2s<\Bigl((N+2s)(p-1)+\frac{N+2s}2-\frac{p}{p+1}N\Bigr)\Bigl(\frac{p}{p+1}-\frac12\Bigr)^{-1}.
$$
Hence, we can choose $\sigma\in (0,\,\frac{N+2s}{2})$ such that
\begin{equation}\label{sigma}
\frac{p+1}{p}\Bigl((N+2s)(p-1)+\frac{N+2s}{2}-\sigma\Bigr)>N,\,\,\,\frac{p}{p+1}-\frac12<\frac{m}{N+2s-m}\sigma.
\end{equation}

Using the fact $ U_{x^i}\le U_{x^1},\,(x\in\Omega_1) $ and  Lemma \ref{lmA.2}, we obtain
\begin{equation}\label{4.10}\begin{split}
&\int_{\r^N}\Big|(\sik U_{x^i})^p-\sik
U_{x^i}^p\Big||\varphi|=k\int_{\Omega_1}\Big|(\sik U_{x^i})^p-\sik
U_{x^i}^p\Big||\varphi|\\
\leq& Ck\int_{\Omega_1}U_{x^1}^{p-1}\sum_{i=2}^kU_{x^i}|\varphi|\\
\leq&
Ck\Big(\int_{\Omega_1}\Big(U_{x^1}^{p-1}\sum_{i=2}^kU_{x^i}\Big)^{\frac{p+1}{p}}\Big)^{\frac{p}{p+1}}\Big(\int_{\Omega_1}|\varphi|^{p+1}\Big)^\frac{1}{p+1}\\
\leq&Ck\Big(\int_{\Omega_1}\Big(\frac{1}{(1+|x-x^1|)^{(N+2s)(p-1)+\frac{N+2s}{2}-\sigma}}\Big(\frac{k}{r}\Big)^{\frac{N+2s}{2}+\sigma}\Big)^{\frac{p+1}{p}}\Big)
^{\frac{p}{p+1}}\Big(\int_{\Omega_1}|\varphi|^{p+1}\Big)^\frac{1}{p+1}\\
=&Ck^{\frac{p}{p+1}}\Big(\frac{k}{r}\Big)^{\frac{N+2s}{2}+\sigma}\Big(\int_{\Omega_1}\Big(\frac{1}{(1+|x-x^1|)^{(N+2s)(p-1)
+\frac{N+2s}{2}-\sigma}}\Big)^\frac{p+1}{p}\Big)^{\frac{p}{p+1}}
\|\varphi\|_s\\
\leq&Ck^{\frac12}\Bigl(\frac{k}{r}\Bigr)^{\frac{N+2s}{2}+\tau}\|\varphi\|_s,\\
\end{split}
\end{equation}
for some small constant  $\tau>0$, where we have used \eqref{sigma}.

 On the other hand, we have
\begin{equation}\label{4.11}\begin{split}
&\int_{\r^N}\Big|(K(x)-1)\sik
U_{x^i}^p\varphi\Big|=k\int_{\r^N}\Big|(K(x)-1)U_{x^1}^p\varphi\Big|\\
\leq&
Ck\Big(\int_{\r^N}|K(x)-1|^{\frac{p+1}{p}}U_{x^1}^{p+1}\Big)^{\frac{p}{p+1}}\Big(\int_{\r^N}|\varphi|^{p+1}\Big)^{\frac{1}{p+1}}\\
\leq&
Ck\|\varphi\|_s\Big(\Bigl(\int_{B_{\frac{r}{2}}(x^1)}|K(x)-1|^{\frac{p+1}{p}} U_{x^1}^{p+1}\Bigr)^{\frac{p}{p+1}}+C\int_{\r^N\setminus
B_{\frac{r}{2}}(x^1)}
U_{x^1}^{p+1}\Big)^{\frac{p}{p+1}}\\
\leq& C\|\varphi\|_s\Big(\frac{k}{r^{m}}+\frac{k}{r^{(N+2s)p-\frac{p}{p+1}N}}\Big)\\
 \leq&C \frac{k^{\frac12}}{r^{\frac{m}2+\tau}}\|\varphi\|_s,
\end{split}
\end{equation}
for some small constant  $\tau>0$, where the last inequality is due to the assumption $\frac{N+2s}{N+2s+1}<m<N+2s$.

Inserting \eqref{4.10}--\eqref{4.11} into \eqref{4.9}, we can complete the proof.
\end{proof}

Now, we are ready  to prove our main theorem. Let
$\omega=\omega(r)$ be the map obtained in Proposition \ref{prop4.2}.
Define
$$F(r)=I(U_r+\omega),\ \ \ \forall\, r\in S_k.$$  It follows from  Lemma~6.1 in \cite{dpw} that
if $r$ is a critical point of $F(r)$, then $U_r+\omega$ is a
solution of \eqref{eq}.

\begin{proof}[\textbf{Proof of Theorem \ref{th3}}]
It follows from Propositions \ref{prop4.2} and  \ref{propA.1} that
\begin{equation}\label{3.9}\begin{split}
F(r)&=I(U_r)+O(\|l\|\|\omega\|+\|\omega\|^2)\\
&=k\Big(A-\frac{B_0
k^{N+2s}}{r^{N+2s}}+\frac{B_1}{r^m}+O\Big(\frac{1}{r^{m+\sigma}}\Big)+O\Bigl(\Bigl(\frac{k}{r}\Bigr)^{N+2s+2\tau}+\frac{1}{r^{m+2\tau}}\Bigr)\Big)\\
&=k\Big(A-\frac{B_0
k^{N+2s}}{r^{N+2s}}+\frac{B_1}{r^m}+O\Big(\frac{1}{r^{m+\sigma}}\Big)\Big).
\end{split}
\end{equation}

Define
$$
F_1(r) :=-\frac{B_0
k^{N+2s}}{r^{N+2s}}+\frac{B_1}{r^m}+O\Big(\frac{1}{r^{m+\sigma}}\Big).
$$
We consider the following maximization problem
$$
\max\limits_{r\in S_k} F_1(r).
$$
Suppose that $\hat r$ is a maximizer, we will prove that $\hat r$ is
an interior point of $S_k$.

We can check that the function
 $$
 G(r)=-\frac{B_0
k^{N+2s}}{r^{N+2s}}+\frac{B_1}{r^m}
$$
has a maximum point
$$
\widetilde{r}=\Big(\frac{B_0(N+2s)}{B_1m}\Big)^{\frac{1}{N+2s-m}}k^{\frac{N+2s}{N+2s-m}}
$$
and
$$
\frac{B_0k^{N+2s}}{\widetilde{r}^{N+2s}}=\frac{B_1}{\widetilde{r}^m}\frac{m}{N+2s}.
$$
 By direct computation, we deduce that
\begin{eqnarray*}
F_1(\widetilde{r})&=&-\frac{B_0
k^{N+2s}}{\widetilde{r}^{N+2s}}+\frac{B_1}{\widetilde{r}^m}+O\Big(\frac{1}{\widetilde{r}^{m+\sigma}}\Big)\\
&=&\frac{B_1}{{\tilde r}^m}\Bigl(1-\frac{m}{N+2s}\Bigr)+O\Big(\frac{1}{\widetilde{r}^{m+\sigma}}\Big)\\
&=&\frac{B_1^{\frac{N+2s}{N+2s-m}}}{B_0^{\frac{m}{N+2s-m}}}\frac{1}{(\frac{N+2s}{m})
^{\frac{N+2s}{N+2s-m}}}\Big(\frac{N+s}{2m}-1\Big)k^{-\frac{(N+2s)m}{N+2s-m}}+O(k^{-\frac{(N+2s)m}{N+2s-m}-\sigma}).
\end{eqnarray*}
On the other hand, we find
\begin{eqnarray*}
&&F_1\Big(\Big(\frac{B_0(N+2s)}{B_1m}-\alpha\Big)^{\frac{1}{N+2s-m}}k^{\frac{N+2s}{N+2s-m}}\Big)\\
&=&-\frac{B_0
k^{N+2s}}{(\frac{N+2s}{m}\frac{B_0}{B_1}-\alpha)^{\frac{N+2s}{N+2s-m}}k^{\frac{(N+2s)^2}{N+2s-m}}}
+\frac{B_1}{(\frac{N+2s}{m}\frac{B_0}{B_1}-\alpha)^{\frac{m}{N+2s-m}}k^{\frac{m(N+2s)}{N+2s-m}}}+O(k^{-\frac{(N+2s)m}{N+2s-m}-\sigma})\\
&=&\frac{B_1^{\frac{N+2s}{N+2s-m}}}{B_0^{\frac{m}{N+2s-m}}}\frac{1}{(\frac{N+2s}{m}-\alpha\frac{B_1}{B_0})
^{\frac{N+2s}{N+2s-m}}}\Big(\frac{N+2s}{m}-\alpha\frac{B_1}{B_0}-1\Big)k^{-\frac{(N+2s)m}{N+2s-m}}+O(k^{-\frac{(N+2s)m}{N+2s-2m}-\sigma})\\
&<&F_1(\widetilde{r})
\end{eqnarray*}
and similarly
\begin{eqnarray*}
&&F_1\Big(\Big(\frac{B_0(N+2s)}{B_1m}+\alpha\Big)^{\frac{1}{N+2s-m}}k^{\frac{N+2s}{N+2s-m}}\Big)\\
&=&\frac{B_1^{\frac{N+2s}{N+2s-m}}}{B_0^{\frac{m}{N+2s-m}}}\frac{1}{(\frac{N+2s}{m}+\alpha\frac{B_1}{B_0})
^{\frac{N+2s}{N+2s-m}}}\Big(\frac{N+2s}{m}+\alpha\frac{B_1}{B_0}-1\Big)k^{-\frac{(N+2s)m}{N+2s-m}}+O(k^{-\frac{(N+s)m}{N+2s-2m}-\sigma})\\
&<&F_1(\widetilde{r}),
\end{eqnarray*}
where we have used the fact that the function
$f(x)=x^{-\frac{N+2s}{N+2s-m}}(x-1)$ attains its maximum at
$x_0=\frac{N+2s}{m}$ if
$x\in[\frac{N+2s}{m}-\alpha,\frac{N+2s}{m}+\alpha]$.

The above estimates implies  that $\hat r$ is indeed an  interior
point of $S_k$.   Thus
$$u_{\hat r}=U_{\hat r}+\omega_{\hat r}$$ is a solution of \eqref{eq}.

At last, we claim that $u_{\hat r}>0$. Indeed, since $\|\omega_{\hat r}\|_s\to 0$ as $k\to\infty$, noticing the fact (see \cite{dpw} for example) that
$$
\|u_{\hat r}\|^2_s=\int_{\r_+^{N+1}}|\nabla \tilde u_{\hat r}|^2 y^{1-2s}dxdy+\int_{\r^N}u_{\hat r}^2dx,
$$
where $\tilde u_{\hat r}(x,y)$ is the $s$-harmonic extension of $u_{\hat r}$ satisfying $\tilde u_{\hat r}(x,0)=u_{\hat r}(x)$, we can use
 the standard argument to verify that $(u_{\hat r})_-=0$ and hence $u_{\hat r}\ge 0$. Since $u_{\hat r}$ solves \eqref{2.4}, we  conclude by using the strong maximum principle that $u_{\hat r}>0$.
\end{proof}

\appendix

\section{{ Energy expansion}}\label{s4}
In this section, we will give the energy expansion for the
approximate solutions. Recall
$$
x^i=\Big(r\cos\frac{2(i-1)\pi}{k},r\sin\frac{2(i-1)\pi}{k}, 0\Big),
\ i=1,\cdots,k,
$$
$$
\Omega_i=\Big\{x=(x',x'')\in \r^2\times
\r^{N-2}:\langle\frac{x'}{|x'|},\frac{x^i}{|x^i|}\rangle\geq
\cos\frac{\pi}{k} \Big\},i=1,2,\cdots,k,
$$
and

$$
I(u)=\frac{1}{2}\int_{\r^{2N}}\frac{|u(x)-u(y)|^2}{|x-y|^{N+2s}}+\frac{1}{2}\int_{\r^N}u^2-\frac{1}{p+1}\int_{\r^N}K(x)|u|^{p+1}.
$$

Firstly, we will introduce Lemma B.1 (\cite{wy2}).

\begin{lemma}\label{lmA.1} For any constant $0<\sigma\leq
\min\{\al,\be\}$, there is a constant $C>0$, such that
$$
\frac{1}{(1+|y-x^i|)^{\al}}\frac{1}{(1+|y-x^j|)^{\be}}\leq
\frac{C}{|x^i-x^j|^{\sigma}}\Big(\frac{1}{(1+|y-x^i|)^{\al+\be-\sigma}}+\frac{1}{(1+|y-x^j|)^{\al+\be-\sigma}}\Big),
$$
where $\al,\ \be\geq 1$ are two constants.
\end{lemma}

Then, we have the following basic estimate:

\begin{lemma}\label{lmA.2} For any $x\in \Omega_1$, and
$\eta\in (1,N+2s]$, there is a constant $C>0$, such that
$$\sum_{i=2}^kU_{x^i}\leq C\frac{1}{(1+|x-x^1|)^{N+2s-\eta}}\frac{k^\eta}{|x_1|^\eta}\leq C\frac{k^\eta}{|x_1|^\eta}.$$
\end{lemma}
\begin{proof}The proof of this lemma is similar to Lemma A.1 in
\cite{wy1}, we sketch the proof below for the sake of completeness.

For any $x\in \Omega_1$, we have for $i\neq 1$,
 $$
 |x-x^i|\geq|x-x^1|,\ \ \forall\
x\in\Omega_1,
$$
which gives $|x-x^i|\geq\frac{1}{2}|x^i-x^1|$ if
$|x-x^1|\geq\frac{1}{2}|x^i-x^1|$. On the other hand, if
$|x-x^1|\leq\frac{1}{2}|x^i-x^1|$, then
$$|x-x^i|\geq|x^i-x^1|-|x-x^1|\geq\frac{1}{2}|x^i-x^1|.$$ So, we
find $$|x-x^i|\geq\frac{1}{2}|x^i-x^1|,\ \ \forall\ x\in\Omega_1.$$
Thus,
\begin{equation*} \begin{split}
 \sum_{i=2}^kU_{x^i}&=
\sum_{i=2}^k\frac{C}{1+|x-x^i|^{N+2s}}=\sum_{i=2}^k\frac{C}{(1+|x-x^i|)^{N+2s}}\\
&=C
\sum_{i=2}^k\frac{1}{(1+|x-x^i|)^\eta}\frac{1}{(1+|x-x^i|)^{N+2s-\eta}}\\
&\leq
C\frac{1}{(1+|x-x^1|)^{N+2s-\eta}}\sum_{i=2}^k\frac{1}{|x^1-x^i|^\eta}.
\end{split}
\end{equation*}
Since
$$
|x^i-x^1|=2|x^1|\sin\frac{(i-1)\pi}{k},\ \ i=2,\cdots,k,
$$
we have
\begin{eqnarray*}
\sum_{i=2}^k\frac{1}{|x^i-x^1|^\eta}&=&\frac{1}{(2|x^1|)^\eta}\sum_{i=2}^k\frac{1}{(\sin\frac{(i-1)\pi}{k})^\eta}\\
& =& \left\{\begin{array}{ll}
        \frac{1}{(2|x^1|)^\eta}\sum_{i=2}^{\frac{k}{2}}\frac{1}{(\sin\frac{(i-1)\pi}{k})^\eta}+\frac{1}{(2|x^1|)^\eta}, & \text{if}\  k\ \text{is \ even},\vspace{2mm}\\
       \frac{1}{(2|x^1|)^\eta}\sum_{i=2}^{[\frac{k}{2}]}\frac{1}{(\sin\frac{(i-1)\pi}{k})^\eta}, & \text{if}\ k\ \text{is \ odd}.
    \end{array}\right.
\end{eqnarray*}

But
$$
0<c'\leq\frac{\sin\frac{(i-1)\pi}{k}}{\frac{(i-1)\pi}{k}}\leq C'',\
\ j=2,\cdots,[\frac{k}{2}].
$$
So, there is a constant $B>0$, such that
$$
\sum_{i=2}^k\frac{1}{|x^i-x^1|^\eta}=\frac{B
k^\eta}{|x^1|^\eta}+O(\frac{k}{|x^1|^\eta}).
$$
Thus, we obtain
$$
\sum_{i=2}^kU_{x^i}\leq
C\frac{1}{(1+|x-x^1|)^{N+2s-\eta}}\frac{k^\eta}{|x^1|^\eta}\leq
C\frac{k^\eta}{|x^1|^\eta}.
$$
\end{proof}

\begin{proposition}\label{propA.1} There is a small constant $\tau >0$, such that
\begin{equation*}\begin{split}
I(U_r)=&k\Big(A-\frac12\sum_{j=2}^k\frac{\tilde B_0}{|x^1-x^j|^{N+2s}}+\frac{B_1}{r^m}+O\Big(\frac{1}{r^{m+\tau}}+\Big(\frac{k}{r}\Big)^{N+2s+\tau}\\
&+\sum_{j=2}^k\frac{B_0}{|x^1-x^j|^{N+2s+\tau}}\Big)\Big)\\
=&k\Big(A-\frac{B_0
k^{N+2s}}{r^{N+2s}}+\frac{B_1}{r^m}+O\Big(\frac{1}{r^{m+\tau}}\Big)\Big),
\end{split}
\end{equation*}
 where
$A=(\frac{1}{2}-\frac{1}{p})\int_{\r^N}U^{p}$, and $B_0,\ B_1>0$ are
positive constants.
\end{proposition}

\begin{proof}
 Using the symmetry,

\begin{equation}\label{A.1} \begin{split}
\langle U_r,U_r\rangle_s+\langle U_r,U_r\rangle_{L^2(\r^N)}&=\sik\sum_{j=1}^{k}\int_{\r^N}U_{x^j}^{p}U_{x^i}\\
&=k\Bigl(\int_{\r^N}U_{x^1}^{p+1}+\sum_{j=2}^{k}\int_{\r^N}U_{x^1}^{p}U_{x^j}\Bigr)\\
 &=k \int_{\r^N}U^{p+1}+k \sum_{i=2}^{k}\int_{\r^N}U_{x^1}^{p}U_{x^i}.\\
\end{split}
\end{equation}
It follows from Lemma \ref{lmA.1} that
\begin{equation*}\label{A.2} \begin{split}
&\sum_{i=2}^k\int_{\r^N}U_{x^1}^pU_{x^i}=C\int_{\r^N}\Big(\frac{1}{1+|x-x^1|^{N+2s}}\Big)^p\sum_{i=2}^k\frac{1}{1+|x-x^i|^{N+2s}}\\
&\le C
\sum_{i=2}^k\frac{1}{|x^1-x^i|^{N+2s}}\int_{\r^N}\frac{1}{(1+|x-x^1|)^{(N+2s)p}}+O\Big(\sum_{i=2}^k\frac{1}{|x^1-x^i|^{N+2s+\tau}}\Big)\\
&=\sum_{i=2}^k\frac{C_1}{|x^1-x^i|^{N+2s}}+O\Big(\sum_{i=2}^k\frac{1}{|x^1-x^i|^{N+2s+\tau}}\Big).
\end{split}
\end{equation*}
However,
\begin{eqnarray*}
\sum_{i=2}^k\int_{\r^N}U_{x^1}^pU_{x^i}&=&\int_{\r^N}\Big(\frac{1}{1+|x-x^1|^{N+2s}}\Big)^p\sum_{i=2}^k\frac{1}{1+|x-x^i|^{N+2s}}\\
&\ge&
\sum_{i=2}^k\int_{B_{\frac{|x^1-x^i|}{2}}(x^1)}\Big(\frac{1}{1+|x-x^1|^{N+2s}}\Big)^p\frac{1}{1+|x-x^i|^{N+2s}}\\
&&+\sum_{i=2}^k
\int_{B_{\frac{|x^1-x^i|}{2}}(x^i)}\Big(\frac{1}{1+|x-x^1|^{N+2s}}\Big)^p\frac{1}{1+|x-x^i|^{N+2s}}\\
&\ge&\sum_{i=2}^k\frac{C_2}{|x^1-x^i|^{N+2s}}+O\Big(\sum_{i=2}^k\frac{1}{|x^1-x^i|^{N+2s+\tau}}\Big).
\end{eqnarray*}

Hence, there exists $B_0'$ (which maybe depend on $k$) in $[C_2,\,\,C_1]$, where $C_1$ and $C_2$ are independent of $k$, such that
\begin{equation}\label{A.2} \begin{split}
&\sum_{i=2}^k\int_{\r^N}U_{x^1}^pU_{x^i}=\sum_{i=2}^k\frac{B_0'}{|x^1-x^i|^{N+2s}}+O\Big(\sum_{i=2}^k\frac{1}{|x^1-x^i|^{N+2s+\tau}}\Big).
\end{split}
\end{equation}

Now, by symmetry, we see
\begin{eqnarray*}
\int_{\r^N}K(x)U_r^{p+1}&=&k\int_{\Omega_1}K(x)U_{x^1}^{p+1}+k(p+1)\int_{\Omega_1}K(x)\sum_{i=2}^kU_{x^1}^pU_{x^i}\\
&& +k\left\{\begin{array}{ll}
         O\Big(\ds\int_{\Omega_1}U_{x^1}^{\frac{p+1}{2}}(\sum_{i=2}^kU_{x^i})^{\frac{p+1}{2}}\Big), & \text{if}\  1<p<2,\vspace{3mm}\\
       O\Big(\ds\int_{\Omega_1}U_{x^1}^{p-1}(\sum_{i=2}^kU_{x^i})^2\Big), & \text{if}\ p\geq2.
    \end{array}\right.
\end{eqnarray*}

For $x\in \Omega_1$, we have $|x-x^i|\geq \frac12|x^i-x^1|$. By Lemma \ref{lmA.1},
\begin{equation*}\begin{split}
\sum_{i=2}^kU_{x^i}&\leq
C\sum_{i=2}^k\frac{1}{(1+|x-x^1|)^{N+2s-\kappa}}\frac{1}{(1+|x-x^i|)^{\kappa}}\\
&\leq C\sum_{i=2}^k
\frac{1}{|x^1-x^i|^{N+2s-\kappa}}\Big(\frac{1}{(1+|x-x^1|)^{\kappa}}+\frac{1}{(1+|x-x^i|)^{\kappa}}\Big)\\
 &\leq
C\sum_{i=2}^k
\frac{1}{|x^1-x^i|^{N+2s-\kappa}}\frac{1}{(1+|x-x^1|)^{\kappa}},\\
\end{split}
\end{equation*}
where $\kappa>0$ satisfies $\min\{\frac{p+1}{2}(N+2s-\kappa),\,2(N+2s-\kappa)\}>N+2s$.
Hence, we get
\begin{eqnarray*}
&&\int_{\Omega_1}U_{x^1}^{\frac{p+1}{2}}(\sum_{i=2}^kU_{x^i})^{\frac{p+1}{2}}\\
&\leq&C\int_{\Omega_1}\frac{1}{(1+|x-x^1|)^{\frac{(N+2s)(p+1)}{2}}}\Big(\sum_{i=2}^k\frac{1}{|x^i-x^1|^{N+2s-\kappa}}\Big)^\frac{p+1}{2}
\frac{1}{(1+|x-x^1|)^{\frac{p+1}{2}\kappa}}\\
&=&C\Big(\sum_{i=2}^k\frac{1}{|x^i-x^1|^{N+2s-\kappa}}\Big)^\frac{p+1}{2}\int_{\Omega_1}\frac{1}{(1+|x-x^1|)^{\frac{(N+2s)(p+1)}{2}+\kappa}}\\
&\le&C\Big(\sum_{i=2}^k\frac{1}{|x^i-x^1|^{N+2s-\kappa}}\Big)^\frac{p+1}{2}\le C\Big(\frac{k}{r}\Big)^{N+2s+\tau}
\end{eqnarray*}
and
\begin{eqnarray*}
&&\int_{\Omega_1}U_{x^1}^{p-1}(\sum_{i=2}^kU_{x^i})^2\\
&\leq&C\int_{\Omega_1}\frac{1}{(1+|x-x^1|)^{(N+2s)(p-1)}}\Big(\sum_{i=2}^k\frac{1}{|x^i-x^1|^{N+2s-\kappa}}\Big)^2\frac{1}{(1+|x-x^1|)^{2\kappa}}\\
&=&C\Big(\sum_{i=2}^k\frac{1}{|x^i-x^1|^{N+2s-\kappa}}\Big)^2\int_{\Omega_1}\frac{1}{(1+|x-x^1|)^{(N+2s)(p-1)+2\kappa}}\\
&=&C\Big(\sum_{i=2}^k\frac{1}{|x^i-x^1|^{N+2s-\kappa}}\Big)^2\le  C\Big(\frac{k}{r}\Big)^{N+2s+\tau}.
\end{eqnarray*}
On the other hand,
\begin{equation*}\begin{split}
\int_{\Omega_1}K(x)U_{x^1}^p\sum_{i=2}^kU_{x^i}&=\int_{\Omega_1}U_{x^1}^p\sum_{i=2}^kU_{x^i}+\int_{\Omega_1}(K(x)-1)U_{x^1}^p\sum_{i=2}^kU_{x^i}.
\end{split}
\end{equation*}
But, from Lemma~\ref{lmA.1} and \eqref{A.2},
\begin{eqnarray*}
&&\int_{\Omega_1}U_{x^1}^p\sum_{i=2}^kU_{x^i}=\int_{\r^N}U_{x^1}^p\sum_{i=2}^kU_{x^i}-\int_{\r^N\setminus\Omega_1}U_{x^1}^p\sum_{i=2}^kU_{x^i}\\
&=&\int_{\r^N}U_{x^1}^p\sum_{i=2}^kU_{x^i}
+O\Big(\Big(\frac{k}{r}\Big)^\sigma\int_{\r^N\setminus\Omega_1}U_{x^1}^{p-\sigma}\sum_{i=2}^kU_{x^i}\Big)\\
&=&\int_{\r^N}U_{x^1}^p\sum_{i=2}^kU_{x^i}+O\Big(\Big(\frac{k}{r}\Big)^\sigma\sum_{i=2}^k\frac{1}{|x^i-x^1|^{N+2s}}
\int_{\r^N\setminus\Omega_1}\Big(U_{x^1}^{p-\sigma}+U^{p-\sigma}_{x^i}\Big)\Big)\\
&=&\sum_{i=2}^k\frac{B_0'}{|x^1-x^i|^{N+2s}}+O\Big(\Big(\frac{k}{r}\Big)^{N+2s+\tau}\Big),
\end{eqnarray*}
where $\sigma>0$ satisfies $p-\sigma>1$.

Moreover, similarly,
\begin{eqnarray*}
&&\int_{\Omega_1}|K(x)-1|U_{x^1}^p\sum_{i=2}^kU_{x^i}\\
&=&\int_{\r^N}|K(x)-1|U_{x^1}^p\sum_{i=2}^kU_{x^i}
-\int_{\r^N\setminus\Omega_1}|K(x)-1|U_{x^1}^p\sum_{i=2}^kU_{x^i}\\
&\le&\int_{B_{\frac{r}{2}}(x^1)}|K(x)-1|U_{x^1}^p\sum_{i=2}^kU_{x^i}+C\int_{\r^N\setminus B_{\frac{r}{2}}(x^1)}U_{x^1}^p\sum_{i=2}^kU_{x^i}
+O\Big(\Big(\frac{k}{r}\Big)^{N+2s+\tau}\Big)\\
&\le&\frac{C}{r^m}\int_{\r^N}U_{x^1}^p\sum_{i=2}^kU_{x^i}+O\Big(\frac{1}{r^\sigma}\sum_{i=2}^k\frac{1}{|x^i-x^1|^{N+2s}}
\int_{\r^N\setminus B_{\frac{r}{2}}(x^1)}\Big(U_{x^1}^{p-\sigma}+U^{p-\sigma}_{x^i}\Big)\Big)\\
&&+O\Big(\Big(\frac{k}{r}\Big)^{N+2s+\tau}\Big)\\
&=&O\Big(\Big(\frac{k}{r}\Big)^{N+2s+\tau}+\frac{1}{r^{m+\tau}}\Big).
\end{eqnarray*}
Hence,
\begin{equation*}\begin{split}
\int_{\Omega_1}K(x)U_{x^1}^p\sum_{i=2}^kU_{x^i}
&=\sum_{j=2}^k\frac{B'_0}{|x^1-x^i|^{N+2s}}+O\Big(\Big(\frac{k}{r}\Big)^{N+2s+\tau}+\frac{1}{r^{m+\tau}}\Big).
\end{split}
\end{equation*}

Finally,
\begin{eqnarray*}
&&\int_{\Omega_1}K(x)U_{x^1}^{p+1}\\
&=&\int_{\r^N}K(x)U_{x^1}^{p+1}-\int_{\r^N\setminus B_{\frac{2\pi r}{k}}(x^1)}K(x)U_{x^1}^{p+1}+\int_{\Omega_1\setminus B_{\frac{2\pi r}{k}}(x^1)}K(x)U_{x^1}^{p+1}\\
&=&\int_{\r^N}K(x)U_{x^1}^{p+1}+O\Big(\int_{\r^N\setminus B_{\frac{2\pi r}{k}}(x^1)}K(x)U_{x^1}^{p+1}\Big)\\
&=&\int_{B_{\frac{r}{2}}(x^1)}K(x)U_{x^1}^{p+1}+\int_{\r^N\setminus B_{\frac{r}2}(x^1)}K(x)U_{x^1}^{p+1}+O\Big(\int_{\r^N\setminus B_{\frac{2\pi r}{k}}(x^1)}K(x)U_{x^1}^{p+1}\Big)\\
&=&\int_{\r^N}U^{p+1}-\frac{B'_1}{r^m}+O\Big(\frac{1}{r^{m+\tau}}\Big)+O\Big(\Big(\frac{k}{r}\Big)^{(N+2s)(p+1)-N}\Bigl)\\
&=& \int_{\r^N}U^{p+1}-\frac{B'_1}{r^m}+O\Big(\frac{1}{r^{m+\tau}}+\Big(\frac{k}{r}\Big)^{N+2s+\tau}\Big)
\end{eqnarray*}
since $(N+2s)(p+1)-N>N+2s$.

So, we have proved
\begin{equation}\label{A.3}
\int_{\r^N}K(x)U_r^{p+1}=k\Big(\int_{\r^N}U^{p+1}+\sum_{j=2}^k\frac{B'_0}{|x^1-x^i|^{N+2s}}-\frac{B'_1}{r^m}
+O\Big(\frac{1}{r^{m+\tau}}+\Big(\frac{k}{r}\Big)^{N+2s+\tau}\Big)\Big).
\end{equation}
Now, inserting \eqref{A.1}--\eqref{A.3} into $I(U_r)$, we complete
the proof.
\end{proof}

\noindent{\bf Acknowledgements:}\,\,\, S. Peng was partially
supported by the fund from NSFC(11125101). W. Long was partially
supported by the NSF of Jiangxi Province (20132BAB211004). J. Yang
was partially supported by the the excellent doctorial dissertation
cultivation grant from Central China Normal University (2013YBZD15).

\end{document}